\documentclass{amsart}

\usepackage{amssymb,amsthm}

\usepackage{a4wide}

\usepackage{hyperref,graphicx}

\title{The Selberg integral and Young books}
\author{Jang Soo Kim}
\address[Jang Soo Kim]{
Department of Mathematics, Sungkyunkwan University, Suwon 440-746,
South Korea}
\email{jangsookim@skku.edu}

\author{Suho Oh}
\address[Suho Oh]{
Department of Mathematics, Uninversity of Michigan, Ann Arbor, USA}
\email{}

\thanks{The first author was supported by Basic Science Research
  Program through the National Research Foundation of Korea (NRF)
  funded by the Ministry of Education (NRF-2013R1A1A2061006).}
\date{\today}
\newtheorem{thm}{Theorem}[section]
\newtheorem{lem}[thm]{Lemma}
\newtheorem{prop}[thm]{Proposition}
\newtheorem{cor}[thm]{Corollary}
\theoremstyle{definition}
\newtheorem{defn}{Definition}[section]

\newtheorem{problem}{Problem}[section]

\theoremstyle{remark}

\newcommand\flr[1]{\left\lfloor #1 \right\rfloor}

\newcommand{\SP}{\mathrm{SP}}
\newcommand{\SB}{\mathrm{SB}}
\newcommand{\RSB}{\mathrm{RSB}}
\newcommand{\YB}{\mathrm{YB}}

\newcommand{\vr}{\mathbf{r}}
\newcommand{\vs}{\mathbf{s}}
\newcommand{\va}{\mathbf{a}}

\begin{document}

\begin{abstract}
  The Selberg integral is an important integral first evaluated by
  Selberg in 1944. Stanley found a combinatorial interpretation of the
  Selberg integral in terms of permutations. In this paper, new
  combinatorial objects ``Young books'' are introduced and shown to
  have a connection with the Selberg integral. This connection gives
  an enumeration formula for Young books. It is shown that special
  cases of Young books become standard Young tableaux of various
  shapes: shifted staircases, squares, certain skew shapes, and
  certain truncated shapes. As a consequence, product formulas for the
  number of standard Young tableaux of these shapes are obtained.
 \end{abstract}

\maketitle

\section{Introduction}

The Selberg integral is the following integral first evaluated by
Selberg \cite{Selberg1944} in 1944:
  \begin{align}
\label{eq:Selberg}
S_n(\alpha,\beta,\gamma) &= \int_0^1\cdots\int_0^1
\prod_{i=1}^n x_i^{\alpha-1} (1-x_i)^{\beta-1}
\prod_{1\le i<j\le n} |x_i-x_j|^{2\gamma} dx_1\cdots dx_n\\ \notag
&=\prod_{j=1}^n \frac{\Gamma(\alpha+(j-1)\gamma)
\Gamma(\beta+(j-1)\gamma)\Gamma(1+j\gamma)}
{\Gamma(\alpha+\beta+(n+j-2)\gamma)\Gamma(1+\gamma)},
  \end{align}
  where $n$ is a positive integer and $\alpha,\beta,\gamma$ are complex numbers
  such that $\mathrm{Re} (\alpha)>0$, $\mathrm{Re} (\beta)>0$, and $\mathrm{Re}
  (\gamma)> -\min\{1/n,\mathrm{Re}(\alpha)/(n-1),\mathrm{Re}(\beta)/(n-1)\}$.
  We refer the reader to Forrester and Warnaar's exposition
  \cite{Forrester2008} for the history and importance of the Selberg
  integral.

  In \cite[Exercise 1.10 (b)]{EC1} Stanley gives a combinatorial
  interpretation of the Selberg integral when the exponents
  $\alpha-1,\beta-1$ and $2\gamma$ are nonnegative integers by
  introducing certain permutations. In this paper we define ``Selberg
  books'' which are essentially a graphical representation of these
  permutations as fillings of certain Young diagrams.  We then define
  ``Young books'' which are special Selberg books. Young books are a
  generalization of both shifted Young tableaux of staircase shape and
  standard Young tableaux of square shape. We show that there is a
  simple relation between the number of Selberg books and the number
  of Young books by finding generating functions for both
  objects. 

  It is well known that the number of standard Young tableaux has a
  nice product formula due to Frame, Robinson, and Thrall
  \cite{Frame1954} in which every factor is at most the size of the
  shape. However, the number of standard Young tableaux of a skew
  shape or a truncated shape may not have such a product formula since
  it may have a large prime factor compared to the size of the shape.
  A truncated shape is a diagram obtained from a usual Young diagram
  in English convention by removing cells in its southwest
  corner. Standard Young tableaux of truncated shapes were recently
  considered by Adin and Roichman \cite{Adin2012}. They showed that
  the number of geodesics between two antipodes in the flip graph of
  triangle-free triangulations is equal to twice the number of
  standard Young tableaux of certain shifted truncated shape. Adin,
  King, and Roichman \cite{Adin2011} and Panova \cite{Panova2012}
  showed that the number of standard Young tableaux of certain
  truncated shapes has a product formula.  As a consequence of our
  formula for the Young books, we obtain some product formulas for the
  number of standard Young tableaux of some skew shapes and truncated
  shapes.

  This paper is organized as follows.  In
  Section~\ref{sec:comb-interpr} we review Stanley's combinatorial
  interpretation of the Selberg integral. In
  Section~\ref{sec:selberg-books-young} we define Selberg books and
  Young books in a simple form which are related to the Selberg
  integral when $\alpha=\beta=1$. We show that there is a simple
  relation between their cardinalities by finding generating functions
  for them. Using this relation and \eqref{eq:Selberg} we obtain a
  formula for the number of Young books.  In
  Section~\ref{sec:extend-selb-books} we define Selberg books and
  Young books in the complete form which are related to the Selberg
  integral without restriction. Results in
  Section~\ref{sec:comb-interpr} are extended here. As a consequence
  we obtain product formulas for the number of standard Young tableaux
  of a truncated shape obtained from a rectangle by removing a
  staircase from the southwest corner, and a skew shape obtained by
  attaching two such truncated shapes. Using generating functions, we
  find another integral expression for the Selberg integral.  In
  Section~\ref{sec:gen-selb-books} we consider generalized Selberg
  books. We find a product formula for the number of standard Young
  tableaux of a truncated shape, which is more general than two
  truncated shapes considered by Panova \cite{Panova2012}.

\section{Stanley's combinatorial interpretation}
\label{sec:comb-interpr}

In this section we review Stanley's combinatorial interpretation of
the Selberg integral in terms of probability when
$r=\alpha-1,s=\beta-1$ and $m=2\gamma$ are nonnegative integers. 

Let $A(n,r,s,m)$ be the following set of letters
\begin{multline*}
A(n,r,s,m) = \{x_i:1\le i\le n\}\cup
\{a_{ij}^{(k)}:1\le i<j\le n,1\le k\le m\} \\
\cup \{b_i^{(k)}:1\le i\le n,1\le k\le r\}\cup
\{c_i^{(k)}:1\le i\le n,1\le k\le s\}.
\end{multline*}
A permutation of $A(n,r,s,m)$ is called a \emph{Selberg permutation}
if the following conditions hold:
\begin{itemize}
\item $x_1,x_2,\dots,x_n$ are in this order,
\item $a_{ij}^{(k)}$ is between $x_i$ and $x_j$ for $1\le i<j\le n$ and $1\le k\le m$,
\item $b_i^{(k)}$ is before $x_i$ for $1\le i\le n$ and $1\le k\le r$,
  and 
\item $c_i^{(k)}$ is after $x_i$ for $1\le i\le n$ and $1\le k\le s$.
\end{itemize}
Let $\SP(n,r,s,m)$ denote the set of Selberg permutations of
$A(n,r,s,m)$.  For example 
\[
b_1^{(1)} x_1 a_{13}^{(1)} b_2^{(1)} a_{13}^{(2)} c_1^{(2)}
c_1^{(1)}
a_{12}^{(2)}a_{12}^{(1)} b_3^{(1)} x_2 a_{23}^{(2)} c_2^{(1)} a_{23}^{(1)} x_3 
c_3^{(1)} c_2^{(2)} c_3^{(2)}
\in \SP(3,1,2,2).
\]

Then we have the following combinatorial interpretation
for the Selberg integral, see \cite[Exercise 1.10 (b)]{EC1}.

\begin{prop}\label{prop:Stanley}
We have
\[
\int_0^1\cdots\int_0^1 \prod_{i=1}^n x_i^r (1-x_i)^s
\prod_{1\le i<j\le n} |x_i-x_j|^{m} dx_1\cdots dx_n
= \frac{n!|\SP(n,r,s,m)|}{((r+s+1)n+mn(n-1)/2)!}.
\]  
\end{prop}


  For a nonnegative integer $n$, we define
\[
n!! = \prod_{j=0}^{\flr{(n-1)/2}} (n-2j).
\]
In other words,
\[
(2k)!! = (2k) (2k-2)\cdots 2, \qquad
(2k-1)!! = (2k-1) (2k-3)\cdots 1.
\]

By \eqref{eq:Selberg}, Proposition~\ref{sec:comb-interpr}, and the
facts $\Gamma(1+n) = n!$ and $\Gamma\left(\frac 12 + n\right)
=\frac{(2n-1)!!}{2^n} \sqrt{\pi}$, we obtain the following formula for
the number of Selberg permutations.

\begin{prop}\label{prop:sp}
We have
\[
|\SP(n,r,s,m)| = \frac{2^n ((r+s+1)n+mn(n-1)/2)!}{n!}
\prod_{j=1}^n 
\frac{(jm)!! (2r+(j-1)m)!!(2s+(j-1)m)!!}{m!!(2r+2s+2+(n+j-2)m)!!}.
\]
\end{prop}
  

\section{$(n,m)$-Selberg books and $(n,m)$-Young books}
\label{sec:selberg-books-young}

In this section we define $(n,m)$-Selberg books which are in natural
bijection with the Selberg permutations $\SB(n,r,s,m)$ when $r=s=0$. We
then define $(n,m)$-Young books which are $(n,m)$-Selberg books with
an additional condition. In the next section we will consider more general
Selberg books and Young books which are related to $\SB(n,r,s,m)$ for
any nonnegative integers $r$ and $s$.

The \emph{shifted staircase of size $n$} is the shifted partition
$(n,n-1,\dots,1)$.  The cell in the $i$th row and $i$th column is
called the \emph{$i$th diagonal cell}.  We will identify the shifted
staircase of size $n$ with its shifted Young diagram as shown in
Figure~\ref{fig:staircase}.

\begin{figure}
  \centering
  \includegraphics{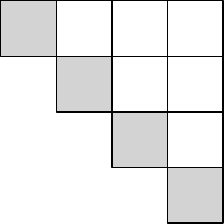} 
  \caption{A shifted staircase of size $4$. The diagonal cells are shaded.}
\label{fig:staircase}
\end{figure}

\begin{defn}
  Let $\lambda^{(1)}, \lambda^{(2)}, \dots,\lambda^{(m)}$ be shifted
  staircases of size $n$. We identify the $i$th diagonal cells of
  $\lambda^{(1)}, \lambda^{(2)}, \dots, \lambda^{(m)}$ for each $1\le
  i\le n$. We call $\lambda^{(i)}$ the \emph{$i$th page}. An
  \emph{(n,m)-Selberg book} is a filling of the $m$-tuple
  $(\lambda^{(1)}, \lambda^{(2)}, \dots,\lambda^{(m)})$ with integers
  $1,2,\dots,n+m\binom n2$ such that in each page the integer in the
  $i$th row and $j$th column with $i\ne j$ is bigger than the integer
  in the $i$th diagonal cell and smaller than the integer in the $j$th
  diagonal cell.  Let $\SB(n,m)$ be the set of $(n,m)$-Selberg books.
\end{defn}

See Figure~\ref{fig:(3,2)-SB} for an example of $(n,m)$-Selberg book. 

\begin{figure}
  \centering
  \includegraphics{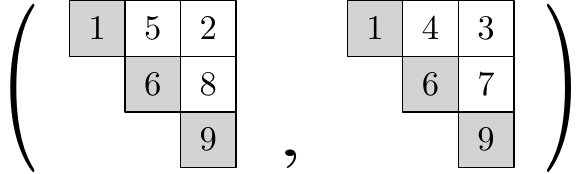} 
  \caption{A $(3,2)$-Selberg book. The diagonal cells are shaded.}
\label{fig:(3,2)-SB}
\end{figure}

There is a natural bijection between $\SB(n,m)$ and $\SP(n,0,0,m)$ as
follows. For $B\in \SB(n,m)$, define the corresponding permutation
$\pi=\pi_1\pi_2\dots\pi_{n+m\binom n2}$ by 
\[
\pi_\ell= \left\{
\begin{array}{ll}
  x_i, & \mbox{if $B$ has $\ell$ in the $i$th diagonal cell,}\\
  a_{ij}^{(k)}, & \mbox{if $B$ has $\ell$ in
    the $i$th row and $j$th column of the $k$th shifted staircase with
    $i\ne j$. }\\
\end{array}
\right.
\]
For instance, the permutation corresponding to the Selberg book in
Figure~\ref{fig:(3,2)-SB} is
\[
x_1 a_{13}^{(1)} a_{13}^{(2)}
a_{12}^{(2)}a_{12}^{(1)} x_2 a_{23}^{(2)} a_{23}^{(1)} x_3.
\]
Thus, by Proposition~\ref{prop:sp} with $r=s=0$, we obtain a formula
for $|\SB(n,m)|$.
\begin{prop}\label{prop:SB(n,m)}
We have
\[
  |\SB(n,m)|=\frac{2^n (n+mn(n-1)/2)!}{n! m!!^n}
  \prod_{j=1}^n \frac{((j-1)m)!!^2(jm)!!}{(2+(n+j-2)m)!!}.
\]
\end{prop}

\begin{defn}
An \emph{$(n,m)$-Young book} is an $(n,m)$-Selberg book with the
additional condition that for each shifted staircase the integers are
increasing along each row and column.  Let $\YB(n,m)$ be the set of
$(n,m)$-Young books.
\end{defn}

Note that an $(n,1)$-Young book is a just standard Young
tableau of shifted staircase shape.  By attaching the two shifted
staircases along the diagonal cells after flipping over the second
shifted staircase, an $(n,2)$-Young book can be thought of as a
standard Young tableau of square shape $(n^n)$, see
Figure~\ref{fig:square}.

\begin{figure}
  \centering
  \includegraphics{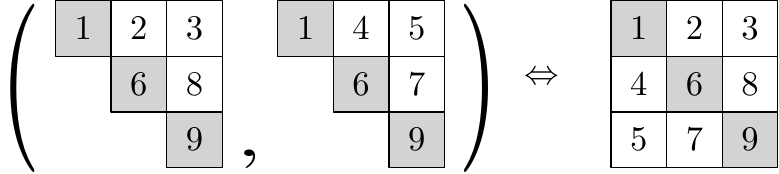} 
  \caption{The correspondence between $(n,2)$-Selberg books and
    standard Young tableaux of square shape $(n^n)$. The diagonal
    cells are shaded.}
\label{fig:square}
\end{figure}

For the rest of this section we will find a simple relation between
the cardinalities of $\SB(n,m)$ and $\YB(n,m)$. 

We define $\SB(n,m;d_1,\dots,d_{n-1})$ and
$\YB(n,m;d_1,\dots,d_{n-1})$ to be, respectively, the set of
$(n,m)$-Selberg books and the set of $(n,m)$-Young books such that the
entries $a_1,a_2,\dots,a_n$ in the diagonal cells satisfy
$d_i=a_{i+1}-a_i-1$ for $i=1,2,\dots,n-1$. Note that since we always
have $a_1=1$ and $a_n=n+m\binom n2$, the numbers $d_1,\dots,d_{n-1}$
determine $a_1,a_2,\dots,a_n$, and vice versa.

\begin{prop}\label{prop:gf_SB}
For nonnegative integers $n$ and $m$, we have
\[
\sum_{d_1,\dots,d_{n-1}\ge 0} |\SB(n,m;d_1,\dots,d_{n-1})|
\frac{t_1^{d_1} \dots t_{n-1}^{d_{n-1}}}{d_1!\dots d_{n-1}!}
=\prod_{1\le i<j\le n} (t_i + t_{i+1} +\dots + t_{j-1})^m.
\]
\end{prop}
\begin{proof}
  We define a \emph{reduced $(n,m)$-Selberg book} to be a filling of
  an $m$-tuple of shifted staircases of size $n$ with integers
  $1,2,\dots,n-1$ with repetition allowed such that the diagonal cells
  are empty and a non-diagonal cell in the $i$th row and $j$th column
  is filled with an integer $k$ satisfying $i\le k < j$.  Let
  $\RSB(n,m;d_1,\dots,d_{n-1})$ denote the set of reduced
  $(n,m)$-Selberg books with $d_1$ 1's, $d_2$ 2's, and so on.  By
  definition we have
\[
\sum_{d_1,\dots,d_{n-1}\ge 0} |\RSB(n,m;d_1,\dots,d_{n-1})|
t_1^{d_1} \dots t_{n-1}^{d_{n-1}}
=\prod_{1\le i<j\le n} (t_i + t_{i+1} +\dots + t_{j-1})^m.
\]

Let $B\in \SB(n,m;d_1,\dots,d_{n-1})$. Then the entries
$a_1,\dots,a_n$ in the diagonal cells of $B$ satisfy
$d_i=a_{i+1}-a_i-1$. Let $B'$ be the reduced $(n,m)$-Selberg book
obtained from $B$ by replacing the $d_i$ integers
$a_i+1,a_i+2,\dots,a_{i+1}-1$ in $B$ with $i$'s for each
$i=1,2,\dots,n-1$. It is easy to see that the map $B\mapsto B'$ is
$1$-to-$d_1!\dots d_{n-1}!$, which finishes the proof.
\end{proof}

By computing the volume of the Gelfand-Tsetlin polytopes in two
different ways, Postnikov \cite{Postnikov2009} showed that
\begin{equation}
  \label{eq:Postnikov}
\sum_{d_1,\dots,d_{n-1}\ge 0} |\YB(n,1;d_1,\dots,d_{n-1})|
\frac{t_1^{d_1} \dots t_{n-1}^{d_{n-1}}}{d_1!\dots d_{n-1}!}
=\prod_{1\le i<j\le n} \frac{t_i + t_{i+1} +\dots + t_{j-1}}{j-i}.
\end{equation}

\begin{thm}\label{thm:SB=YB}
  We have
  \begin{equation}
    \label{eq:SB=YB1}
|\SB(n,m;d_1,\dots,d_{n-1})| = 
\left(1!2!\cdots (n-1)!\right)^m \cdot |\YB(n,m;d_1,\dots,d_{n-1})|,
  \end{equation}
  \begin{equation}
    \label{eq:SB=YB2}
|\SB(n,m)| = \left(1!2!\cdots (n-1)!\right)^m \cdot |\YB(n,m)|.    
  \end{equation}
\end{thm}
\begin{proof}
  Since \eqref{eq:SB=YB2} is obtained from \eqref{eq:SB=YB1} by
  summing over all $d_1,\dots,d_{n-1}$, it suffices to prove
  \eqref{eq:SB=YB1}. By Proposition~\ref{prop:gf_SB} and
  \eqref{eq:Postnikov}, we have
  \begin{equation}
    \label{eq:1}
|\SB(n,1;d_1,\dots,d_{n-1})| = 
1!2!\cdots (n-1)! \cdot |\YB(n,1;d_1,\dots,d_{n-1})|.
  \end{equation}
  Hence, the theorem is true for the case $m=1$.  We now consider for
  an arbitrary $m$.

  Let $a_1,a_2,\dots,a_n$ be the integers satisfying $a_0=1$ and
  $d_i=a_{i+1}-a_i-1$ for $i=1,2,\dots,n-1$.  For a set $X$ of $\binom
  n2$ integers, let $\SB_X(n,1;d_1,\dots,d_{n-1})$ be the set of
  fillings of a shifted staircase of size $n$ with integers in
  $X\cup\{a_1,\dots,a_{n}\}$ so that the $i$th diagonal cell is
  filled with $a_i$ and a non-diagonal cell in the $i$th row and $j$th
  column is filled with an integer $k$ satisfying $a_i < k < a_j$. 
  By considering each shifted staircase separately we get
\begin{equation}
  \label{eq:2}
|\SB(n,m;d_1,\dots,d_{n-1})| = \sum_{X_1,\dots,X_m}
\prod_{i=1}^m |\SB_{X_i}(n,1;d_1,\dots,d_{n-1})|,
\end{equation}
where the sum is over all subsets $X_1,\dots,X_m$ of $\{1,2,\dots,
n+m\binom{n}2\}\setminus\{a_1,\dots,a_{n}\}$ such that $|X_i| = \binom
n2$ for all $i$, and
\[
X_1\cup \dots \cup X_m = \left\{1,2,\dots, n+m\binom{n}2 \right\}
\setminus\{a_1,\dots,a_{n}\}.
\]

Similarly we can define
$\YB_X(n,1;d_1,\dots,d_{n-1})$ and obtain
\begin{equation}
  \label{eq:3}
|\YB(n,m;d_1,\dots,d_{n-1})| = \sum_{X_1,\dots,X_m}
\prod_{i=1}^m |\YB_{X_i}(n,1;d_1,\dots,d_{n-1})|.
\end{equation}
For given $X_i$, we have
\[
|\SB_{X_i}(n,1;d_1,\dots,d_{n-1})|=|\SB(n,1;d'_1,\dots,d'_{n-1})|,
\]
\[
|\YB_{X_i}(n,1;d_1,\dots,d_{n-1})|=|\YB(n,1;d'_1,\dots,d'_{n-1})|,
\]
for the same $d'_1,\dots,d'_{n-1}$. 
Thus by \eqref{eq:1} we have
\[
|\SB_{X_i}(n,1;d_1,\dots,d_{n-1})| = 
1!2!\cdots (n-1)! \cdot |\YB_{X_i}(n,1;d_1,\dots,d_{n-1})|.
\]
Applying the above equation to \eqref{eq:2} and \eqref{eq:3} we get
\eqref{eq:SB=YB1}. 
\end{proof}

By Proposition~\ref{prop:gf_SB} and \eqref{eq:SB=YB1} we obtain the
following generalization of Postnikov's result \eqref{eq:Postnikov}. 

\begin{cor}\label{cor:gf_YB}
We have
\[
\sum_{d_1,\dots,d_{n-1}\ge 0} |\YB(n,m;d_1,\dots,d_{n-1})|
\frac{t_1^{d_1} \dots t_{n-1}^{d_{n-1}}}{d_1!\dots d_{n-1}!}
=\left(\prod_{1\le i<j\le n} \frac{t_i + t_{i+1} +\dots + t_{j-1}}{j-i} \right)^m.
\]  
\end{cor}

We note that Corollary~\ref{cor:gf_YB} can also be proved directly
from \eqref{eq:Postnikov} using \eqref{eq:3}.

By Proposition~\ref{prop:SB(n,m)} and \eqref{eq:SB=YB2} we get the number of
$(n,m)$-Young books. 

\begin{cor}\label{cor:YB(n,2t)}
  We have
\[
|\YB(n,m)| = \frac{2^n (n+mn(n-1)/2)!}{n! m!!^n}
\prod_{j=1}^n \frac{((j-1)m)!!^2(jm)!!}{ (j-1)!^{m} (2+(n+j-2)m)!!}.
\]
\end{cor} 

If $m=1$ in Corollary~\ref{cor:YB(n,2t)}, then we get the hook length
formula for the number of standard Young tableaux of shifted staircase
shape of size $n$. If $m=2$ in Corollary~\ref{cor:YB(n,2t)}, the we
get the hook length formula for the number of standard Young tableaux
of square shape $(n^n)$.  This gives a semi-combinatorial proof of the
Selberg integral for $\alpha=\beta=1$ and $\gamma \in \{1/2, 1\}$.

We can obtain a combinatorial proof of the Selberg integral formula
when $r=\alpha-1,s=\beta-1$ and $m=2\gamma$ are nonnegative integers
if we solve the following two problems.

\begin{problem}
  Find a combinatorial proof of Theorem~\ref{thm:SB=YB}.
\end{problem}

\begin{problem}
  Find a combinatorial proof of Corollary~\ref{cor:YB(n,2t)}. 
\end{problem}

One can consider Young books of shape $(\lambda^{(1)},
\dots,\lambda^{(m)})$ for shifted Young diagrams $\lambda^{(i)}$ with
the same number of rows. However, in this case we do not seem to have
a nice product formula.  For instance, the number of Young books of
shape
\[
((6,2,1),(5,4,1),(5,2,1),(4,2,1))
\]
 is equal to
\[
2^4 \cdot 3 \cdot 5^2 \cdot 7 \cdot 17 \cdot 19 \cdot 23 \cdot 1649819.
\]


\section{$(n,\vr,\vs)$-Selberg books and $(n,\vr,\vs)$-Young books}
\label{sec:extend-selb-books}

For a nonnegative integer $r$, a \emph{composition} of $r$ is a
sequence $\vr=(r_1,r_2,\dots,r_m)$ of nonnegative integers summing to
$r$. In this case we write $\vr\vDash r$ and say that the
\emph{length} of $\vr$ is $m$.

In this section, for nonnegative integers $n,r,s,m$ and compositions
$\vr\vDash r$ and $\vs\vDash s$ of length $m$, we define
$(n,\vr,\vs)$-Selberg books and $(n,\vr,\vs)$-Young books which are
related to $\SB(n,r,s,m)$.

An \emph{$(n,r,s)$-staircase} is the diagram obtained from an
$(r+n)\times(n+s)$ rectangle by removing the cells below the diagonal
cells, where the cell in the $(i+r)$th row and $i$th column is called
the $i$th diagonal cell.  An \emph{$(n,r,s)^-$-staircase} is
the diagram obtained from an $(n,r,s)$-staircase by removing the
$r\times s$ rectangle at the northeast corner. See
Figure~\ref{fig:(n,r,s)}.

\begin{figure}
  \centering
  \includegraphics[scale=0.7]{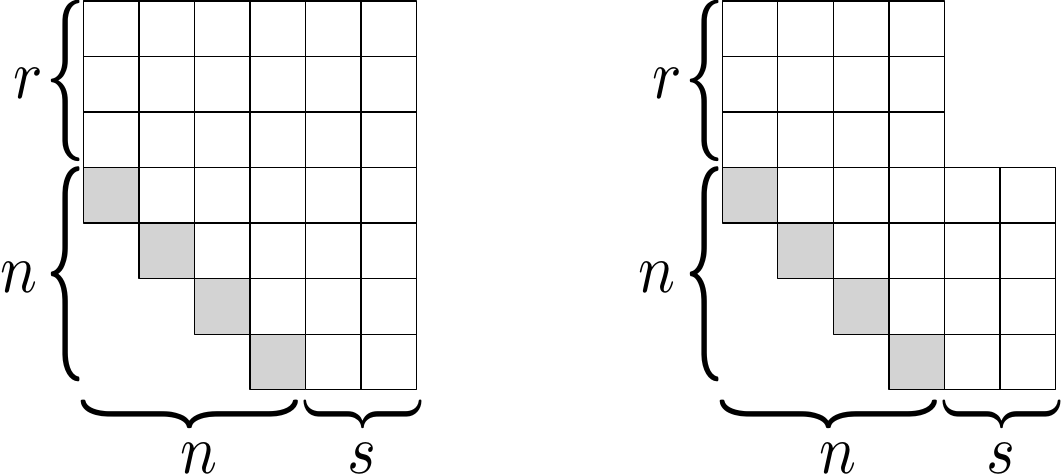} 
  \caption{An $(n,r,s)$-staircase on the left and an
    $(n,r,s)^-$-staircase on the right. The diagonal cells are shaded.}
\label{fig:(n,r,s)}
\end{figure}

\begin{defn}
  Let $\vr=(r_1,\dots,r_m)\vDash r$ and $\vs=(s_1,\dots,s_m)\vDash
  s$. For $1\le i\le m$, let $\lambda^{(i)}$ be a
  $(n,r_i,s_i)^-$-staircase.  We identify the $i$th diagonal cells of
  $\lambda^{(1)},\lambda^{(2)},\dots,\lambda^{(m)}$ for each $1\le
  i\le n$. We call $\lambda^{(i)}$ the \emph{$i$th page}. A
  \emph{$(n,\vr,\vs)^-$-Selberg book} is a filling of $(\lambda^{(1)},
  \dots, \lambda^{(m)})$ with $1,2,\dots,(r+s+1)n+m\binom n2$ such
  that the integer in a non-diagonal cell of each page is bigger than
  the integer in the diagonal cell of the same row and smaller than
  the integer in the diagonal cell of the same column.  Let
  $\SB^-(n,\vr,\vs)$ denote the set of $(n,\vr,\vs)^-$-Selberg books.
\end{defn}

The bijection between $\SB(n,m)$ and $\SP(n,0,0,m)$ in
Section~\ref{sec:selberg-books-young} can easily be extended to a
bijection between $\SB^-(n,\vr,\vs)$ and $\SP(n,r,s,m)$. Notice that
the cardinality of $\SB^-(n,\vr,\vs)$ depends only on $n,r,s$. Thus we
obtain the following proposition.

\begin{prop}\label{prop:SB=SP}
  Let $\vr=(r_1,\dots,r_m)\vDash r$ and $\vs=(s_1,\dots,s_m)\vDash
  s$. Then
\[
|\SB^-(n,\vr,\vs)| = |\SP(n,r,s,m)|.
\]  
\end{prop}

Now we define $(n,\vr,\vs)$-Selberg books.  
\begin{defn}
  Let $\vr=(r_1,\dots,r_m)\vDash r$ and $\vs=(s_1,\dots,s_m)\vDash
  s$. For $1\le i\le m$, let $\lambda^{(i)}$ be an
  $(n,r_i,s_i)$-staircase. We identify the $i$th diagonal cells of
  $\lambda^{(1)},\lambda^{(2)},\dots,\lambda^{(m)}$ for each $1\le
  i\le n$.  An \emph{
    $(n,\vr,\vs)$-Selberg book} is a filling of $(\lambda^{(1)},
  \dots, \lambda^{(m)})$ with $1,2,\dots,(r+s+1)n+m\binom
  n2+\sum_{i=1}^m r_is_i$ such that in each page the integer in a
  non-diagonal cell is bigger than the integer in the diagonal cell of
  the same row and smaller than the integer in the diagonal cell of
  the same column.  Let $\SB(n,\vr,\vs)$ denote the set of
  $(n,\vr,\vs)$-Selberg books.
\end{defn}

There is a simple relation between $|\SB(n,\vr,\vs)|$ and
$|\SB^-(n,\vr,\vs)|$. 

\begin{prop}\label{prop:SB=SB-}
  Let $\vr=(r_1,\dots,r_m)\vDash r$ and $\vs=(s_1,\dots,s_m)\vDash s$.
  Then
\[
|\SB(n,\vr,\vs)| = 
|\SB^-(n,\vr,\vs)| \frac{\left((r+s+1)n+m\binom n2+\sum_{i=1}^m r_is_i\right)!}
{\left((r+s+1)n+m\binom n2\right)!}. 
\]
\end{prop}
\begin{proof}
  This follows from the observation that there are no restrictions on
  the entries of the cells in $(n,\vr,\vs)$-Selberg books which are
  not in $(n,\vr,\vs)^-$-Selberg books.
\end{proof}

\begin{defn}
  For $\vr=(r_1,\dots,r_m)\vDash r$ and $\vs=(s_1,\dots,s_m)\vDash s$,
  we define an \emph{$(n,\vr,\vs)$-Young book} to be an
  $(n,\vr,\vs)$-Selberg book such that in each page the entries are
  increasing from left to right and from top to bottom.  Let
  $\YB(n,\vr,\vs)$ denote the set of $(n,\vr,\vs)$-Young books.  We
  also define $\SB(n,\vr,\vs;d_0,d_1,\dots,d_{n})$ and
  $\YB(n,\vr,\vs;d_0,d_1,\dots,d_{n})$ to be, respectively, the set of
  $(n,\vr,\vs)$-Selberg books and the set of $(n,\vr,\vs)$-Young books
  whose diagonal entries $a_1,\dots,a_n$ satisfy $d_i=a_{i+1}-a_i-1$
  for $i=0,1,2,\dots,n$, where $a_0=1$ and $a_{n+1} = (r+s+1)n+m\binom
  n2+\sum_{i=1}^m r_is_i +1$.
\end{defn}

The following lemma is an immediate consequence of the definition of
Selberg books and Young books.

\begin{lem}\label{lem:freeze}
  Suppose that $d_1,d_2,\dots,d_{n-1}$ is a sequence of nonnegative
  integers such that $d_{k+1}=1, d_{k+2}=2,\dots,d_{k+\ell-1}=\ell-1$
  for some $k,\ell\ge0$. Then, for any $B\in
  \YB(n,1;d_1,d_2,\dots,d_{n-1})$, the entries in rows
  $k+1,k+2,\dots,k+\ell$ and columns $k+1,k+2,\dots,k+\ell$ are
  completely determined by $d_1,\dots,d_{n-1}$. More precisely, for
  $1\le i,j\le \ell$, if $x$ is the entry in the $(k+1)$st diagonal
  cell, which is determined by $d_1,\dots,d_{n-1}$, then the entry in
  row $k+i$ and column $k+j$ is $x+\binom {j-1}2+i$.

  Moreover, if $B\in \SB(n,1;d_1,d_2,\dots,d_{n-1})$ and $x$ is the
  entry in the $(k+1)$st diagonal cell, then the entries in column
  $k+j$ and in rows $k+1,k+2,\dots,k+j-1$ form a permutation of
  $x+\binom {j-1}2+1, x+\binom {j-1}2+2,\dots,x+\binom {j-1}2+j-1$.
\end{lem}

Figure~\ref{fig:freeze} illustrates the situation in
Lemma~\ref{lem:freeze}.

\begin{figure}
  \centering
  \includegraphics[scale=1.3]{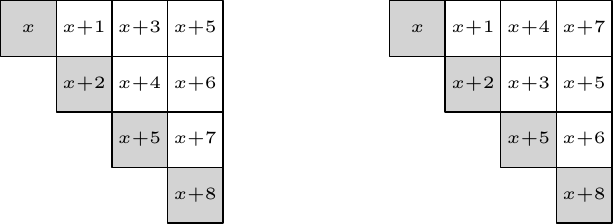} 
  \caption{The diagram on the left shows the typical form of the
    entries in rows $k+1,k+2,\dots,k+\ell$ and columns
    $k+1,k+2,\dots,k+\ell$ of $B\in \YB(n,1;d_1,d_2,\dots,d_{n-1})$
    when $d_{k+1}=1, d_{k+2}=2,\dots,d_{k+\ell-1}=\ell-1$. The diagram
    on the right shows that, in the case of
    $B\in\SB(n,1;d_1,d_2,\dots,d_{n-1})$, for $1\le j\le \ell$, the
    non-diagonal entries in column $k+j$ and below row $k$ are
    obtained by permuting those in the same cells of the diagram on
    the left.  In this example, we have $\ell=4$.}
\label{fig:freeze}
\end{figure}

\begin{prop}\label{prop:SB=YB}
  Let $\vr=(r_1,\dots,r_m)\vDash r$ and $\vs=(s_1,\dots,s_m)\vDash s$.
  Then we have
\[
|\SB(n,\vr,\vs;d_0,d_1,\dots,d_{n})|  
=|\YB(n,\vr,\vs;d_0,d_1,\dots,d_{n})|
\prod_{i=1}^m \frac{1!2!\cdots (n+r_i+s_i-1)!}{1!2!\cdots (r_i-1)!
1!2!\cdots (s_i-1)!}.
\]
\end{prop}
\begin{proof}
  We will prove this only for the case $m=1$. For $m\ge2$, we can use
  the same idea as in the proof of Theorem~\ref{thm:SB=YB}.  Let $m=1,
  \vr=(r)$, and $\vs=(s)$.  Then by Lemma~\ref{lem:freeze} we have
\begin{align*}
|\SB(n,\vr,\vs;d_0,\dots,d_{n})|
&=|\SB(n+r+s,1;1,2,\dots,r-1,d_0,\dots,d_{n}, 1,2,\dots,s-1)|\\
&\qquad\times 1!2!\cdots(r-1)! 1!2!\cdots(s-1)!,\notag \\ 
|\YB(n,\vr,\vs;d_0,\dots,d_{n})|
&=|\YB(n+r+s,1;1,2,\dots,r-1,d_0,\dots,d_{n}, 1,2,\dots,s-1)|.
\end{align*}
By the above equations and \eqref{eq:SB=YB1}, we get the desired
formula for the case $m=1$.
\end{proof}

By Propositions~\ref{prop:sp}, \ref{prop:SB=SP}, \ref{prop:SB=SB-},
and \ref{prop:SB=YB}, we get a formula for $|\YB(n,\vr,\vs)|$. 

\begin{thm}\label{thm:n,vr,vs}
  Let $\vr=(r_1,\dots,r_m)\vDash r$ and $\vs=(s_1,\dots,s_m)\vDash s$.
  Then
  \begin{multline*}
|\YB(n,\vr,\vs)| = \left((r+s+1)n+m\binom n2+\sum_{i=1}^m r_is_i\right)!
\prod_{i=1}^m \frac{1!2!\cdots (r_i-1)!1!2!\cdots (s_i-1)!}
{1!2!\cdots (n+r_i+s_i-1)!}\\
\times  \frac{2^n}{n!} \prod_{j=1}^n 
\frac{(jm)!! (2r+(j-1)m)!!(2s+(j-1)m)!!}{m!!(2r+2s+2+(n+j-2)m)!!}.
  \end{multline*}
\end{thm}

For two partitions $\lambda=(\lambda_1,\dots,\lambda_k)$ and
$\mu=(\mu_1,\dots,\mu_\ell)$, the \emph{skew shape} $\lambda/\mu$ is
defined to be the set-theoretic difference $\lambda - \mu$ of their
Young diagrams. We define the \emph{truncated shape}
$\lambda\backslash\mu$ to be the diagram obtained from the  Young
diagram of $\lambda$ by removing the $\mu_i$ cells from the left in
the $(k+1-i)$th row for $i=1,2,\dots,\ell$. See Figure~\ref{fig:skew}.

\begin{figure}
  \centering
  \includegraphics{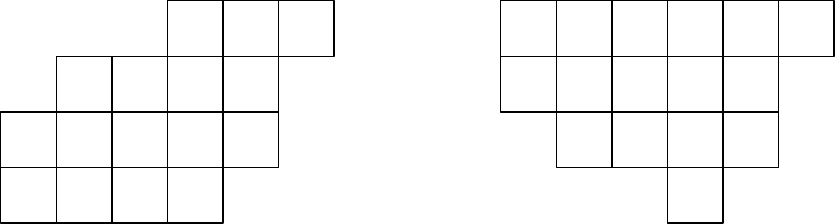} 
  \caption{The skew shape $\lambda/\mu$ on the left and the truncated
    shape $\lambda\backslash\mu$ on the right for $\lambda=(6,5,5,4)$
    and $\mu=(3,1)$.}
\label{fig:skew}
\end{figure}

Notice that, when $m=1$, an $(n,(r),(s))$-Young book is the same as a
standard Young tableau of truncated shape $\lambda\backslash\mu$ for
$\lambda=((n+s)^{r+n})$ and $\mu=(n-1,n-2,\dots,1)$. In this case we
obtain the following collolary from Theorem~\ref{thm:n,vr,vs}.

\begin{cor}
  The number of standard Young tableaux of truncated shape
\[
((n+s)^{r+n})\backslash (n-1,n-2,\dots,1)
\]
is 
\[
\left((r+s+1)n+\binom n2+rs\right)!
\frac{2^n F(r)F(s)}{n! F(n+r+s)}\prod_{j=1}^n 
\frac{(j)!! (2r+j-1)!!(2s+j-1)!!}{(2r+2s+n+j)!!},
\]
where
$F(k) = 1!2!\cdots (k-1)!$. 
\end{cor}

Panova \cite{Panova2012} also found a product formula for the number
in the above corollary. In the next section we find a product formula
for the number of standard Young tableaux of a more general shape. 

When $m=2$, by attaching the two pages along the diagonal cells, an
$(n,(r_1,r_2),(s_1,s_2))$-Young book can be thought of as a standard
Young tableau of skew shape $\lambda/\mu$ for 
\begin{equation}
  \label{eq:lm}
\lambda=((r_2+n+s_1)^{r_1+n}, (r_2+n)^{s_2}), \qquad 
\mu=(r_2^{r_1}).
\end{equation}
See Figure~\ref{fig:skew2} for such a construction.

\begin{figure}
  \centering
  \includegraphics[scale=0.7]{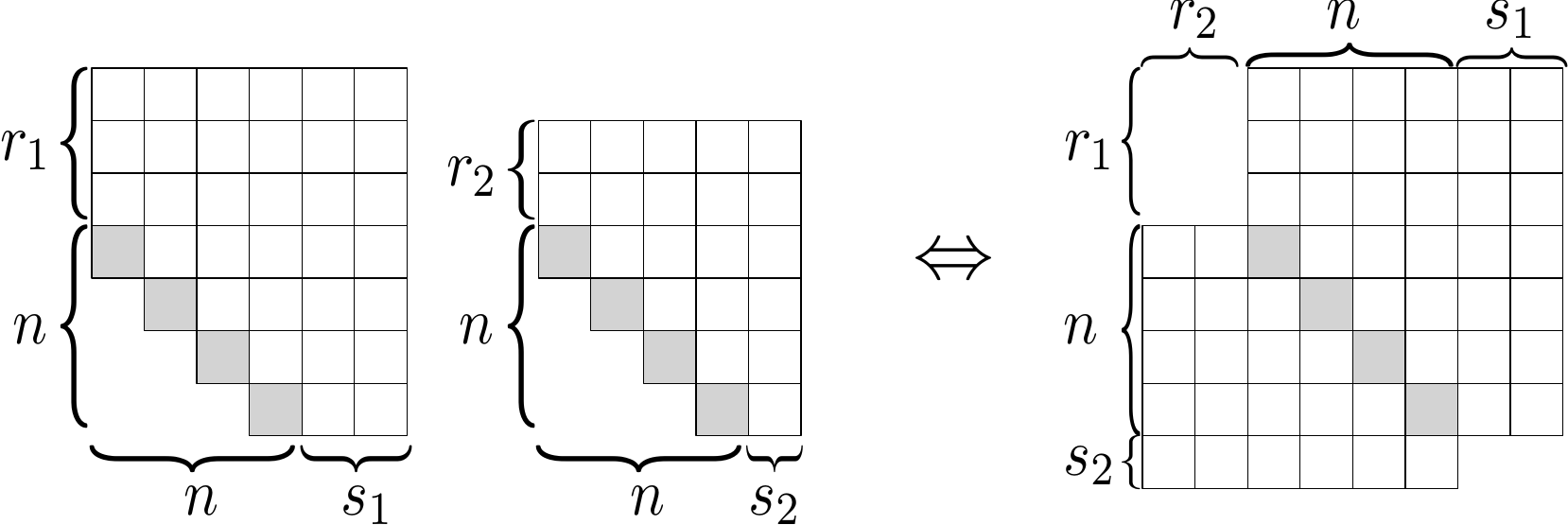} 
  \caption{The skew shape $\lambda/\mu$ on the right is obtained by
    attaching an $(n,r_1,r_2)$-staircase and an
    $(n,r_2,s_2)$-staircase along the diagonal cells. The diagonal
    cells are shaded and the $(n,r_2,s_2)$-staircase is flipped when attached.}
\label{fig:skew2}
\end{figure}

\begin{cor}\label{cor:lambda/mu}
  Let $\lambda$ and $\mu$ be the partitions given in \eqref{eq:lm}
  whose diagram is drawn on the right in Figure~\ref{fig:skew2}.  Then
  the number of standard Young tableaux of skew shape $\lambda/\mu$ is
\[
\frac{2^n\left((r+s)n+n^2+r_1s_1+r_2s_2\right)! F(r_1)F(r_2)F(s_1)F(s_2)}
{n! F(n+r_1+s_1)F(n+r_2+s_2)}
\prod_{j=1}^n \frac{(2j)!! (2r+2j-2)!!(2s+2j-2)!!}{(2r+2s+2n+2j-2)!!}.
\]
where
$F(k) = 1!2!\cdots (k-1)!$. 
\end{cor}

Notice that the skew shape $\lambda/\mu$ in
Corollary~\ref{cor:lambda/mu} is obtained from a rectangle by removing
a smaller rectangle both from its northwest corner and southeast
corner. One may ask if there is a product formula for the number of
standard Young tableaux of any skew shape obtained in this way.  If
$\lambda=(7,7,7,7,7,5,5)$ and $\mu=(4,4)$, then the number of standard
Young tableaux of $\lambda/\mu$ has a factor of $9173$.  Thus, in
general, we cannot expect a product formula for the number of standard
Young tableaux of such a skew shape.  

There is a formula for the number of standard Young tableaux of skew
shape as a determinant, see \cite[7.16.3 Corollary]{EC2}. It would be
interesting to prove Corollary~\ref{cor:lambda/mu} using the
determinantal formula.

By the same arguments as in the previous section, one can prove the
following two propositions.

\begin{prop}\label{prop:gf_SB2}
  Let $\vr=(r_1,\dots,r_m)\vDash r$ and $\vs=(s_1,\dots,s_m)\vDash s$.
  Then we have
   \begin{multline*}
\sum_{d_0,d_1,\dots,d_{n}\ge 0} |\SB(n,\vr,\vs;d_0,d_1,\dots,d_{n})|  
\frac{t_0^{d_1} t_1^{d_1} \dots t_{n}^{d_{n}}}{d_0!d_1!\dots d_{n}!}
\\=\prod_{i=1}^n (t_0 + t_{1} +\dots + t_{i-1})^r 
(t_i + t_{i+1} +\dots + t_{n})^s\\
\times \prod_{i=1}^m (t_0+t_1+\dots+t_n)^{r_is_i}
\prod_{1\le i<j\le n} \left(t_i + t_{i+1} +\dots + t_{j-1}\right)^m.
  \end{multline*}
\end{prop}

\begin{prop}\label{prop:gf_YB2}
  Let $\vr=(r_1,\dots,r_m)\vDash r$ and $\vs=(s_1,\dots,s_m)\vDash s$.
  Then we have
  \begin{multline*}
\sum_{d_0,d_1,\dots,d_{n}\ge 0} |\YB(n,\vr,\vs;d_0,d_1,\dots,d_{n})|  
\frac{t_0^{d_1} t_1^{d_1} \dots t_{n}^{d_{n}}}{d_0!d_1!\dots d_{n}!}
\\=\prod_{i=1}^n (t_0 + t_{1} +\dots + t_{i-1})^r 
(t_i + t_{i+1} +\dots + t_{n})^s
\prod_{1\le i<j\le n} \left(t_i + t_{i+1} +\dots + t_{j-1}\right)^m\\
\times \prod_{i=1}^m (t_0+t_1+\dots+t_n)^{r_is_i}
\frac {1!2!\cdots (r_i-1)!1!2!\cdots (s_i-1)!}{1!2!\cdots
  (n+r_i+s_i-1)!}.
  \end{multline*}
\end{prop}

Using Proposition~\ref{prop:gf_SB2} we can obtain another integral
expression for the Selberg integral.  First, note that
\[
\int_0^{\infty} x^n e^{-x} dx = n!  .
\]
Thus 
\[
|\SB(n,r,s,m)|= \sum_{d_0,d_1,\dots,d_{n}\ge 0} |\SB(n,r,s,m;d_0,d_1,\dots,d_{n})|
\] 
is equal to
\[
\int_0^{\infty} \dots \int_0^{\infty}
\sum_{d_0,d_1,\dots,d_{n}\ge 0} |\SB(n,r,s,m;d_0,d_1,\dots,d_{n})|
\frac{t_0^{d_0}t_1^{d_1} \dots t_{n}^{d_{n}}}{d_0!d_1!\dots d_{n}!}
e^{-t_0-t_1-\dots-t_{n}} dt_0dt_1 \dots dt_{n}.
\]
Using Propositions~\ref{prop:Stanley} and \ref{prop:gf_SB2} we get the
following.

\begin{prop}
We have
  \begin{multline*}
\int_0^{\infty} \dots \int_0^{\infty}
\prod_{i=1}^n (t_0 + t_{1} +\dots + t_{i-1})^r 
(t_i + t_{i+1} +\dots + t_{n})^s\\
\times \prod_{1\le i<j\le n} (t_i + t_{i+1} +\dots + t_{j-1})^m
e^{-t_0-t_1-\dots-t_{n}} dt_0 dt_1 \dots dt_{n}\\
=\frac{((r+s+1)n+mn(n-1)/2)!}{n!}
\int_0^1\cdots\int_0^1 \prod_{i=1}^n x_i^r (1-x_i)^s
\prod_{1\le i<j\le n} |x_i-x_j|^{m} dx_1\cdots dx_n.
  \end{multline*}
\end{prop}

We note that it is also possible to prove the above proposition using
the change of variables.

\section{Generalized Selberg books and Young books}
\label{sec:gen-selb-books}

In this section we generalize Selberg books and Young books so that
a diagonal cell can be a bigger square. 

Let $\va=(a_1,\dots,a_n)\vDash a$. An \emph{$(\va,r,s)$-staircase} is
the diagram obtained from the truncated shape $\lambda\backslash\mu$
by merging the cells in rows
\begin{equation}
  \label{eq:rows}
r+a_1+\cdots+a_{i-1}+1,
r+a_1+\cdots+a_{i-1}+2,\dots,r+a_1+\cdots+a_{i-1}+a_i,
\end{equation}
and columns
\begin{equation}
  \label{eq:columns}
a_1+\cdots+a_{i-1}+1,
a_1+\cdots+a_{i-1}+2,\dots,a_1+\cdots+a_{i-1}+a_i,
\end{equation}
into a single cell, called the \emph{$i$th diagonal cell}, where
\[
\lambda=((a+s)^{(r+a)}),\qquad 
\mu=((a_1+\dots+a_{n-1})^{a_n} , (a_1+\dots+a_{n-2})^{a_{n-1}},\dots,
a_1^{a_2}).
\]
We will consider that the $i$th diagonal cell is contained in every row
whose row index is in \eqref{eq:rows}, and in every column whose
column index is in \eqref{eq:columns}. 
An \emph{$(\va,r,s)^-$-staircase} is the diagram obtained from
an $(\va,r,s)$-staircase by removing the $r\times s$ rectangle in the
northeast corner. See Figure~\ref{fig:(a,r,s)-staircase}.

\begin{figure}
  \centering
\includegraphics[scale=0.7]{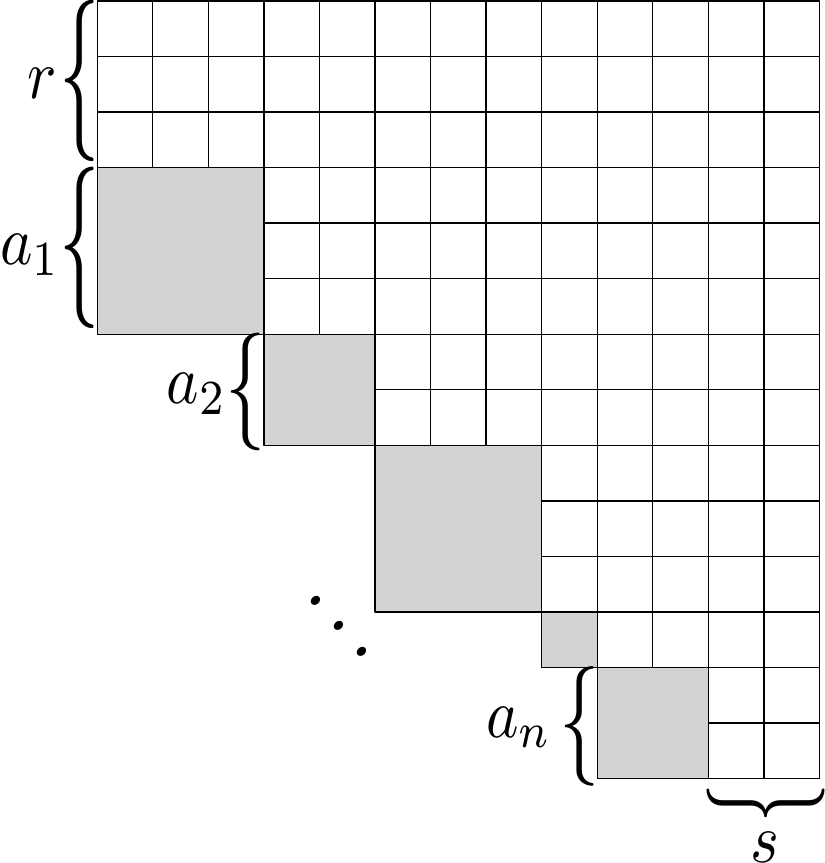}\qquad
  \includegraphics[scale=0.7]{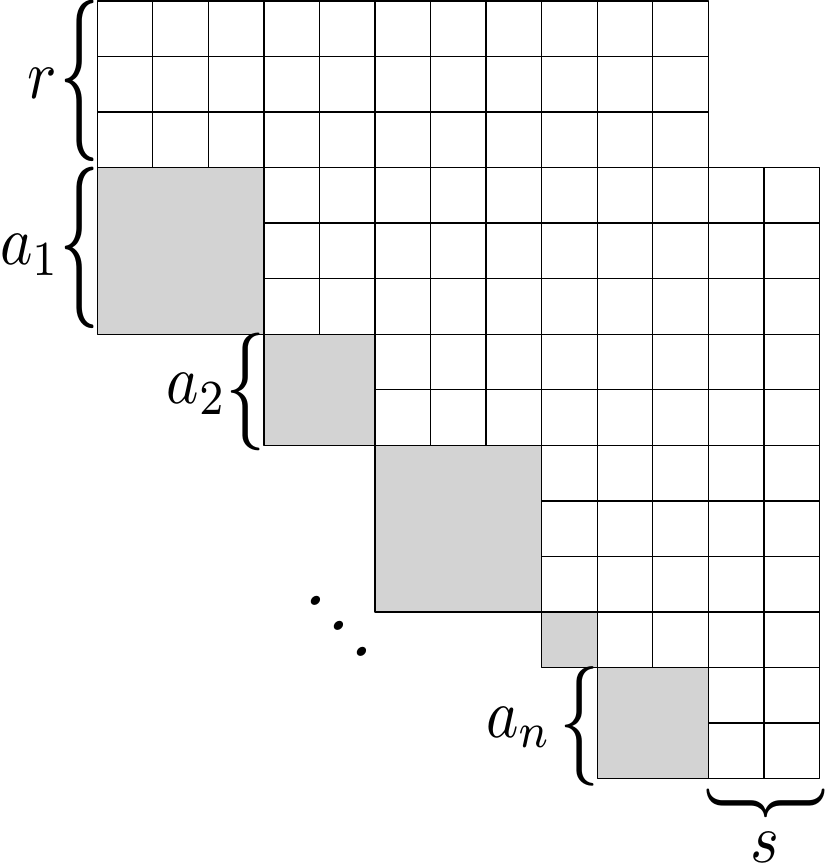}
  \caption{An $(\va,r,s)$-staircase on the left and an
    $(\va,r,s)^-$-staircase on the right. The diagonal cells are shaded.}
\label{fig:(a,r,s)-staircase}
\end{figure}

Throughout this section we will use the following notation.  Let
$\va=(a_1,\dots,a_n)\vDash a$, $\vr=(r_1,\dots,r_m)\vDash r$,
$\vs=(s_1,\dots,s_m)\vDash s$, and
\[
N = n +a(r+s)+m\sum_{1\le i<j\le n} a_i a_j
+ \sum_{i=1}^m r_i s_i,
\]
\[
N^- = n +a(r+s)+m\sum_{1\le i<j\le n} a_i a_j. 
\]

\begin{figure}
  \centering
  \includegraphics{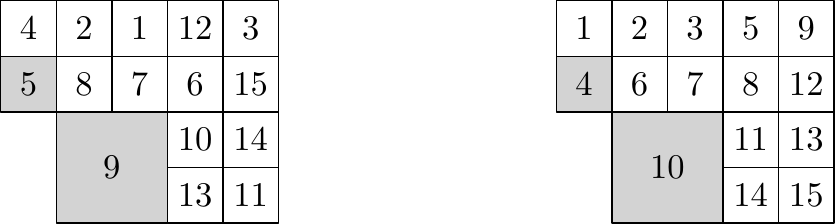}
  \caption{An $(\va,\vr,\vs)$-Selberg book on the left and an
    $(\va,\vr,\vs)$-Young book on the right, where $\va=(1,2),
    \vr=(1), \vs=(2)$. The diagonal cells are
    shaded.}
\label{fig:SBYB}
\end{figure}

\begin{defn}
  For $1\le i\le m$, let $\lambda^{(i)}$ be an
  $(\va,r_i,s_i)$-staircase.  We identify the $i$th diagonal cells of
  $\lambda^{(1)},\lambda^{(2)},\dots,\lambda^{(m)}$ for all $1\le i\le
  n$.  An \emph{$(\va,\vr,\vs)$-Selberg book} is a filling of
  $(\lambda^{(1)}, \dots, \lambda^{(m)})$ with $1,2,\dots,N$ such that
  the integer in a non-diagonal cell is bigger than the integer in the
  diagonal cell of the same row and smaller than the integer in the
  diagonal cell of the same column. See Figure~\ref{fig:SBYB}.  Let
  $\SB(\va,\vr,\vs)$ denote the set of $(\va,\vr,\vs)$-Selberg books.

  Now, for $1\le i\le m$, let $\mu^{(i)}$ be an
  $(\va,r_i,s_i)^-$-staircase.  We identify the $i$th diagonal cell of
  $\mu^{(1)},\mu^{(2)},\dots,\mu^{(m)}$ for each $1\le i\le n$.  An
  \emph{$(\va,\vr,\vs)^-$-Selberg book} is a filling of $(\mu^{(1)},
  \dots, \mu^{(m)})$ with $1,2,\dots,N^-$ such that the integer in a
  non-diagonal cell is bigger than the integer in the diagonal cell of
  the same row and smaller than the integer in the diagonal cell of
  the same column.  Let $\SB^-(\va,\vr,\vs)$ denote the set of
  $(\va,\vr,\vs)^-$-Selberg books.
\end{defn}
 
There is a simple relation between $|\SB(\va,\vr,\vs)|$ and
$|\SB^-(\va,\vr,\vs)|$. 

\begin{prop}\label{prop:SB=SB-2}
We have
\[
|\SB(\va,\vr,\vs)| = |\SB^-(\va,\vr,\vs)| \frac{N!}{(N^-)!}.
\]  
\end{prop}
\begin{proof}
  The proof is similar to that of Proposition~\ref{prop:SB=SB-}.
\end{proof}

By the same idea as in Proposition~\ref{prop:Stanley}, we obtain the
following Proposition.

\begin{prop}\label{prop:stanley2}
We have
\[
\frac{n!}{(N^-)!} |\SB^-(\va,\vr,\vs)|
=\int_{0}^1\cdots \int_{0}^1
\prod_{i=1}^n x_i^{ra_i} (1-x_i)^{sa_i} \prod_{1\le i<j\le n}
|x_i-x_j|^{ma_ia_j}
dx_1\cdots dx_n.
\]  
\end{prop}

\begin{defn}
  We define an $(\va,\vr,\vs)$-Young book to be an
  $(\va,\vr,\vs)$-Selberg book such that in each page entries are
  increasing from left to right in each row and from top to bottom in
  each column. See Figure~\ref{fig:SBYB}.  Let $\YB(\va,\vr,\vs)$
  denote the set of $(\va,\vr,\vs)$-Young books.  We also define
  $\SB(\va,\vr,\vs;d_0,d_1,\dots,d_{n})$ and
  $\YB(\va,\vr,\vs;d_0,d_1,\dots,d_{n})$ to be, respectively, the set
  of $(\va,\vr,\vs)$-Selberg books and the set of
  $(\va,\vr,\vs)$-Young books whose diagonal entries $a_1,\dots,a_n$
  satisfy $d_i=a_{i+1}-a_i-1$ for $i=0,1,2,\dots,n$, where $a_0=1$ and
  $a_{n+1} =N+1$.
\end{defn}

There is a simple relation between
$|\SB(\va,\vr,\vs)|$ and $|\YB(\va,\vr,\vs)|$.

\begin{prop}\label{prop:SB=YB3}
We have
\[
|\SB(\va,\vr,\vs;d_0,d_1,\dots,d_n)|  = |\YB(\va,\vr,\vs;d_0,d_1,\dots,d_n)|  
\prod_{i=1}^m \frac{F(a+r_i+s_i)}{F(a_1)\cdots F(a_n) F(r_i) F(s_i)},
\]
where
$F(k) = 1!2!\cdots (k-1)!$. 
\end{prop}
\begin{proof}
  The proof is similar to that of Proposition~\ref{prop:SB=YB}. We
  will prove this only for the case $m=1$. For $m\ge2$, we can use the
  same idea as in the proof of Theorem~\ref{thm:SB=YB}.

Let $m=1, \vr=(r), \vs=(s)$. By Lemma~\ref{lem:freeze}, we have
\begin{multline*}
  |\SB(\va,(r),(s);d_0,\dots,d_{n})|\\
=|\SB(a+r+s,1;1,2,\dots,r-1,d_0,1,2,\dots,a_1-1,d_1,1,2,\dots,a_2-1,
\dots,\\d_{n-1},1,2,\dots,a_n-1,d_n,1,2,\dots,s-1)|
\cdot F(r)F(a_1)F(a_2)\dots F(a_n) F(s),
\end{multline*}
\begin{multline*}
  |\YB(\va,(r),(s);d_0,\dots,d_{n})|\\
=|\YB(a+r+s,1;1,2,\dots,r-1,d_0,1,2,\dots,a_1-1,d_1,1,2,\dots,a_2-1,
\dots,\\d_{n-1},1,2,\dots,a_n-1,d_n,1,2,\dots,s-1)|.
\end{multline*}
By the above equations and \eqref{eq:SB=YB1}, we get the desired
formula for the case $m=1$.
\end{proof}

If $\va=(k^n)$, then we can evaluate $|\YB(\va,\vr,\vs)|$. 

\begin{cor}
  Let $\va=(k^n)$. Then
  \begin{multline*}
|\YB(\va,\vr,\vs)|  = \frac{2^n \left((kr+ks+1)n+k^2m\binom n2 +
    \sum_{i=1}^m r_i s_i\right) !}{n!}
\prod_{i=1}^m \frac{F(k)^n F(r_i) F(s_i)}{F(kn+r_i+s_i)} \\
\times  \prod_{j=1}^n 
\frac{(j k^2 m)!! (2kr+(j-1)k^2m)!!(2ks+(j-1)k^2m)!!}
{(k^2m)!!(2kr+2ks+2+(n+j-2)k^2m)!!}.
  \end{multline*}
\end{cor}
\begin{proof}
By Propositions~\ref{prop:SB=SB-2}, \ref{prop:stanley2}, and
\ref{prop:SB=YB3}, we have
  \begin{multline*}
|\YB(\va,\vr,\vs)|  = \frac{\left((kr+ks+1)n+k^2m\binom n2 +
    \sum_{i=1}^m r_i s_i\right)!}{n!}
\prod_{i=1}^m \frac{F(a_1)\cdots F(a_n) F(r_i) F(s_i)}{F(a+r_i+s_i)} \\
\times \int_{0}^1\cdots \int_{0}^1
\prod_{i=1}^n x_i^{kr} (1-x_i)^{ks} \prod_{1\le i<j\le n}
|x_i-x_j|^{k^2 m} dx_1\cdots dx_n.
  \end{multline*}
We can now use the Selberg integral formula
\eqref{eq:Selberg} with $\alpha=kr+1,\beta=ks+1$, and $\gamma=2k^2m$,
which finishes the proof.
\end{proof}

\begin{figure}
  \centering
  \includegraphics[scale=0.7]{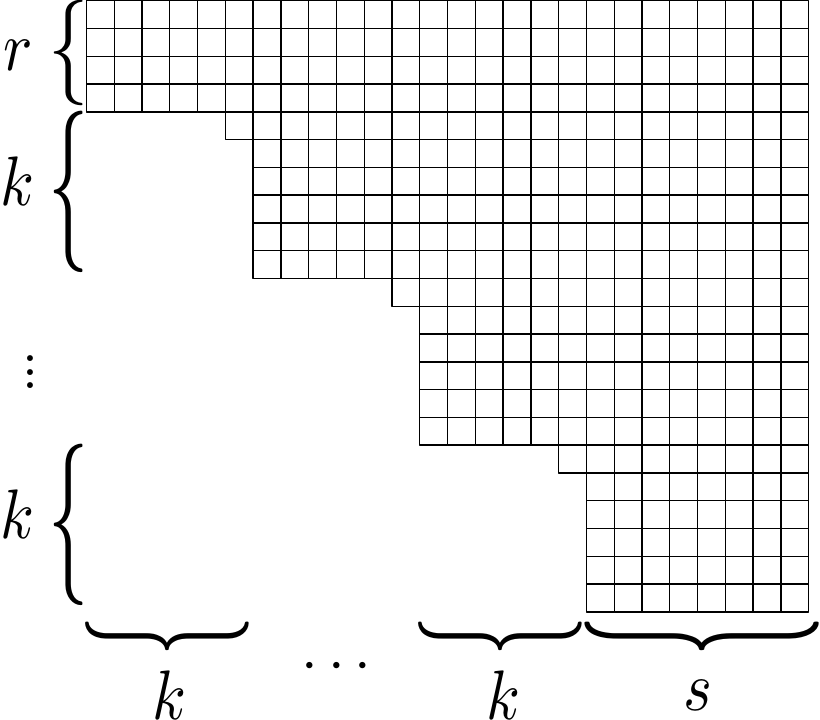}
  \caption{The truncated shape $\lambda\backslash\mu$ for 
$\lambda=((kn+s)^{r+kn})$ and 
$\mu=((kn)^{k-1},kn-1,(kn-k)^{k-1},kn-k-1,\dots,,k^{k-1},k-1)$.}
\label{fig:trunc}
\end{figure}

\begin{figure}
  \centering
  \includegraphics{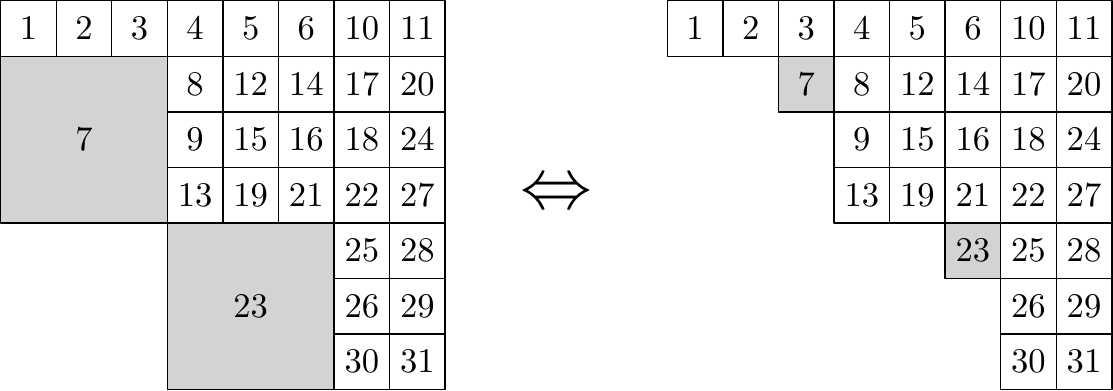}
  \caption{The correspondance between $((k^n),(r),(s))$-Young books
    and standard Young tableaux of truncated shape in
    Figure~\ref{fig:trunc}. }
\label{fig:panova}
\end{figure}

If $\va=(k^n),\vr=(r),\vs=(s)$, then by replacing the each diagonal
cell by a $1\times 1$ cell located at the northeast corner of the
diagonal cell, we can consider an $(\va,\vr,\vs)$-Young book as a
standard Young tableau of truncated shape $\lambda\backslash\mu$ shown
in Figure~\ref{fig:trunc}.  See Figure~\ref{fig:panova} for the
illustration of this correspondence.  Thus we get the following
corollary.

\begin{cor}\label{cor:panova}
  The number of standard Young tableaux of truncated shape in
  Figure~\ref{fig:trunc} is equal to
  \begin{multline*}
\frac{2^n \left((kr+ks+1)n+k^2\binom n2 +
    rs\right) !}{n!}
\frac{F(k)^n F(r) F(s)}{F(kn+r+s)} \\
\times  \prod_{j=1}^n 
\frac{(j k^2)!! (2kr+(j-1)k^2)!!(2ks+(j-1)k^2)!!}
{(k^2)!!(2kr+2ks+2+(n+j-2)k^2)!!},
  \end{multline*}
where $F(n)=1!2!\dots(n-1)!$. 
\end{cor}

Panova \cite{Panova2012} found a formula for the number of standard
Young tableaux of truncated shape $(n^m)\backslash(k-1,k-2,\dots,1)$
and $(n^m)\setminus (k^{k-1},k-1)$. Both of these truncated shapes are
special cases of the truncated shape in Corollary~\ref{cor:panova}.


\end{document}